\documentclass[reqno]{amsart}%
\usepackage{amssymb,graphicx,setspace}
\usepackage{ifpdf}
\ifpdf
 \usepackage[hyperindex,pagebackref]{hyperref}%
\else
 \expandafter\ifx\csname dvipdfm\endcsname\relax
 \usepackage[hypertex,hyperindex,pagebackref]{hyperref}%
 \else
 \usepackage[dvipdfm,hyperindex,pagebackref]{hyperref}%
 \fi
\fi
\allowdisplaybreaks[4]
\theoremstyle{plain}
\newtheorem{thm}{Theorem}[section]
\newtheorem{cor}{Corollary}[section]
\newtheorem{lem}{Lemma}[section]
\theoremstyle{remark}
\newtheorem{rem}{Remark}[section]
\DeclareMathOperator{\td}{d}
\numberwithin{equation}{section}

\begin{document}

\title[Integral representation and properties of Bernoulli numbers]
{An integral representation, some inequalities, and complete monotonicity of Bernoulli numbers of the second kind}

\author[F. Qi]{Feng Qi}
\address[Qi]{Department of Mathematics, School of Science, Tianjin Polytechnic University, Tianjin City, 300387, China}
\email{\href{mailto: F. Qi <qifeng618@gmail.com>}{qifeng618@gmail.com}, \href{mailto: F. Qi <qifeng618@hotmail.com>}{qifeng618@hotmail.com}, \href{mailto: F. Qi <qifeng618@qq.com>}{qifeng618@qq.com}}
\urladdr{\url{http://qifeng618.wordpress.com}}

\author[X.-J. Zhang]{Xiao-Jing Zhang}
\address[Zhang]{The 59th Middle School, Jianxi District, Luoyang City, Henan Province, 471000, China; Department of Mathematics, School of Science, Tianjin Polytechnic University, Tianjin City, 300387, China}
\email{\href{mailto: X.-J. Zhang <xiao.jing.zhang@qq.com>}{xiao.jing.zhang@qq.com}}

\begin{abstract}
In the paper, the authors discover an integral representation, some inequalities, and complete monotonicity of Bernoulli numbers of the second kind.
\end{abstract}



\thanks{This paper was typeset using \AmS-\LaTeX}

\maketitle

\section{Introduction}

In number theory, Bernoulli numbers of the second kind $b_n$ for $n\in\{0\}\cup\mathbb{N}$ may be generated by
\begin{equation}\label{bernoulli-second-dfn}
\frac{x}{\ln(1+x)}=\sum_{n=0}^\infty b_nx^n.
\end{equation}
They are also known as Cauchy numbers of the first kind (see~\cite[p.~294]{Comtet-Combinatorics-74}), Gregory coefficients, or logarithmic numbers. The first few Bernoulli numbers of the second kind $b_n$ are
\begin{equation}
\begin{aligned}\label{Berno-5-values}
b_0&=1, & b_1&=\frac12, & b_2&=-\frac1{12}, & b_3&=\frac1{24}, & b_4&=-\frac{19}{720}, & b_5&=\frac3{160}.
\end{aligned}
\end{equation}

The first main result of this paper is the following integral representation of $b_n$ for $n\in\mathbb{N}$.

\begin{thm}\label{thm-Bernoulli-Integ-Form}
Bernoulli numbers of the second kind $b_n$ may be represented as
\begin{equation}\label{Bernoulli-2rd-int}
b_n=(-1)^{n+1}\int_1^\infty\frac{1}{\{[\ln(t-1)]^2+\pi^2\}t^{n}}\td t, \quad n\in\mathbb{N}.
\end{equation}
\end{thm}

Recall from~\cite[p.~108, Definition~4]{widder} that a sequence $\{\mu_n\}_{0\le n\le\infty}$ is said to be completely monotonic if its elements are non-negative and its successive differences are alternatively non-negative, that is
\begin{equation}
(-1)^k\Delta^k\mu_n\ge0
\end{equation}
for $n,k\ge0$, where
\begin{equation}
\Delta^k\mu_n=\sum_{m=0}^k(-1)^m\binom{k}{m}\mu_{n+k-m}.
\end{equation}
Recall from~\cite[p.~163, Definition~14a]{widder} that a completely monotonic sequence $\{a_n\}_{n\ge0}$ is minimal if it ceases to be completely monotonic when $a_0$ is decreased.
\par
Let $\lambda=(\lambda_1,\lambda_2,\dotsc,\lambda_{n})\in\mathbb{R}^{n}$ and $\mu=(\mu_1,\mu_2,\dotsc, \mu_{n})\in\mathbb{R}^{n}$. A sequence $\lambda$ is said to be majorized by $\mu$ \textup{(}in symbols $\lambda\preceq \mu$\textup{)} if $\sum_{\ell=1}^k \lambda_{[\ell]}\le\sum_{\ell=1}^k \mu_{[\ell]}$ for $k=1,2,\dotsc,n-1$ and $\sum_{\ell=1}^n \lambda_\ell=\sum_{\ell=1}^n\mu_\ell$, where $\lambda_{[1]}\ge \lambda_{[2]}\ge \dotsm \ge \lambda_{[n]}$ and $\mu_{[1]}\ge \mu_{[2]}\ge\dotsm \ge \mu_{[n]}$ are rearrangements of $\lambda$ and $\mu$ in a descending order.
A sequence $\lambda$ is said to be strictly majorized by $\mu$ $($in symbols $\lambda \prec \mu)$ if $\lambda$ is not a permutation of $\mu$.
\par
Basing on Theorem~\ref{thm-Bernoulli-Integ-Form}, the following inequalities and properties of Bernoulli numbers of the second kind $b_n$ are discovered.

\begin{thm}\label{Ber-minimal-thm}
The infinite sequence $\{(-1)^{n}b_{n+1}\}_{n\ge0}$ is completely monotonic and minimal.
\end{thm}

\begin{thm}\label{b(n)-matrix-thm}
Let $m\in\mathbb{N}$ and let $n$ and $a_k$ for $1\le k\le m$ be nonnegative integers. Then
\begin{equation}\label{matrix-1f}
\bigl|(a_k+a_j)!b_{a_k+a_j+1}\bigr|_m\ge0
\end{equation}
and
\begin{equation}\label{matrix-2f}
\bigl|(-1)^{a_k+a_j}{(a_k+a_j)!}b_{a_k+a_j+1}\bigr|_m\ge0,
\end{equation}
where $|a_{kj}|_m$ denotes a determinant of order $m$ with elements $a_{kj}$.
\end{thm}

\begin{thm}\label{lambda-mu-thm}
Let $m\in\mathbb{N}$ and let $\lambda$ and $\mu$ be two $m$-tuples of nonnegative numbers such that $\lambda\preceq\mu$. Then
\begin{equation}\label{lambda-mu-eq}
\Biggl|\prod_{\ell=1}^m {\lambda_\ell!}b_{\lambda_\ell+1}\Biggr|\le \Biggl|\prod_{\ell=1}^m {\mu_\ell!}b_{\mu_\ell+1}\Biggr|.
\end{equation}
\end{thm}

\begin{cor}\label{b(n)-log-Convex-thm}
The infinite sequence $\{(-1)^nn!b_{n+1}\}_{n\ge0}$ is logarithmically convex.
\end{cor}

\section{Lemmas}

For prove our main results, we need the following two integral representations.

\begin{lem}[{\cite[p.~2130]{Berg-Pedersen-ball-PAMS}}]
Let $\mathbb{C}$ be the set of all complex numbers and let
\begin{equation}
\ln z=\ln|z|+i\arg z,
\end{equation}
be the principal branch of the holomorphic extension of $\ln x$ from the open half-line $(0,\infty)$ to the cut plane
$
\mathcal{A}=\mathbb{C}\setminus(-\infty,0],
$
where $-\pi<\arg z<\pi$ and $i=\sqrt{-1}\,$.
The function $\frac1{\ln(1+z)}$ for $z\in\mathbb{C}\setminus(-\infty,0]$ has the integral representation
\begin{equation}\label{recip-ln(1+z)}
\frac1{\ln(1+z)}=\frac1{z}+\int_1^\infty\frac1{[\ln(t-1)]^2+\pi^2} \frac{\td t}{z+t}.
\end{equation}
\end{lem}

\begin{lem}\label{thm-zhang-li-qi}
Let $\mathbb{C}$ be the set of all complex numbers. The function
\begin{equation}\label{F(z)}
F(z)=
\begin{cases}
 \dfrac{z}{(1+z)\ln(1+z)}, & z\in\mathbb{C}\setminus(-\infty,-1]\setminus\{0\}\\
 1, & z=0
\end{cases}
\end{equation}
has the integral representation
\begin{equation}\label{thm-zhang-li-qi-eq}
F(z)=\int_0^\infty \frac{t+1}{t[(\ln t)^2 +\pi^2]}\frac{\td t}{t+1+z},\quad z\in\mathbb{C}\setminus(-\infty,-1].
\end{equation}
\end{lem}

The proof of Lemma~\ref{thm-zhang-li-qi} will be carried out in Section~\ref{sec-proof} below.

\section{Proofs of theorems}

Now we start out to prove Theorems~\ref{thm-Bernoulli-Integ-Form} to~\ref{lambda-mu-thm} and Corollary~\ref{b(n)-log-Convex-thm} as follows.

\begin{proof}[First proof of Theorem~\ref{thm-Bernoulli-Integ-Form}]
By~\eqref{recip-ln(1+z)}, we have
\begin{equation}\label{frac(x)(ln(1+x))-eq1}
\frac{x}{\ln(1+x)}=1+\int_1^\infty\frac1{[\ln(t-1)]^2+\pi^2} \frac{x}{x+t}\td t
\end{equation}
and
\begin{equation}
\begin{split}\label{deriv-int-expression}
\biggl[\frac{x}{\ln(1+x)}\biggr]^{(k)}&=\int_1^\infty\frac1{[\ln(t-1)]^2+\pi^2} \biggl(\frac{x}{x+t}\biggr)^{(k)}\td t\\
&=\int_1^\infty\frac1{[\ln(t-1)]^2+\pi^2} \biggl(1-\frac{t}{x+t}\biggr)^{(k)}\td t\\
&=(-1)^{k+1}k!\int_1^\infty\frac{t}{[\ln(t-1)]^2+\pi^2} \frac1{(x+t)^{k+1}}\td t
\end{split}
\end{equation}
for $k\in\mathbb{N}$. On the other hand, by~\eqref{bernoulli-second-dfn}, we also have
\begin{equation}\label{deriv-dfn-second}
\biggl[\frac{x}{\ln(1+x)}\biggr]^{(k)}=\sum_{n=k}^\infty b_n\frac{n!}{(n-k)!}x^{n-k}.
\end{equation}
Combining~\eqref{deriv-int-expression} with~\eqref{deriv-dfn-second} leads to
\begin{equation*}
\sum_{n=k}^\infty b_n\frac{n!}{(n-k)!}x^{n-k}=(-1)^{k+1}k!\int_1^\infty\frac{t}{[\ln(t-1)]^2+\pi^2} \frac1{(x+t)^{k+1}}\td t.
\end{equation*}
Letting $x\to0^+$ on both sides of the above equation produces
\begin{equation*}
k!b_k=(-1)^{k+1}k!\int_1^\infty\frac1{[\ln(t-1)]^2+\pi^2} \frac1{t^{k}}\td t.
\end{equation*}
Thus, the formula~\eqref{Bernoulli-2rd-int} is proved.
\end{proof}

\begin{proof}[Second proof of Theorem~\ref{thm-Bernoulli-Integ-Form}]
By the integral representation~\eqref{thm-zhang-li-qi-eq}, we have
\begin{equation}\label{frac(x)(ln(1+x))-eq2}
\frac{x}{\ln(1+x)}=\int_1^\infty\dfrac{t}{(t-1)\{[\ln(t-1)]^2+\pi^2\}}\frac{1+x}{x+t}\td t
\end{equation}
and
\begin{equation}
\begin{split}\label{thm-Bernoulli-2rd-eq}
\biggl[\frac{x}{\ln(1+x)}\biggr]^{(k)} &=\int_1^\infty\dfrac{t}{(t-1)\{[\ln(t-1)]^2+\pi^2\}}\biggl(\frac{1+x}{x+t}\biggr)^{(k)}\td t\\
&=\int_1^\infty\dfrac{t}{(t-1)\{[\ln(t-1)]^2+\pi^2\}}\biggl(1+\frac{1-t}{x+t}\biggr)^{(k)}\td t\\
&=(-1)^{k+1}k!\int_1^\infty\dfrac{t}{[\ln(t-1)]^2+\pi^2}\frac{1}{(x+t)^{k+1}}\td t\\
\end{split}
\end{equation}
for $k\in\mathbb{N}$. Combining~\eqref{thm-Bernoulli-2rd-eq} with~\eqref{deriv-dfn-second} leads to
\begin{equation}\label{combine-two-eq-bernoulli2}
\sum_{n=k}^\infty b_n\frac{n!}{(n-k)!}x^{n-k}=(-1)^{k+1}k!\int_1^\infty\frac{t}{[\ln(t-1)]^2+\pi^2} \frac1{(x+t)^{k+1}}\td t.
\end{equation}
Letting $x\to0^+$ on both sides of~\eqref{combine-two-eq-bernoulli2} yields the formula~\eqref{Bernoulli-2rd-int}. The proof of Theorem~\ref{thm-Bernoulli-Integ-Form} is complete.
\end{proof}

\begin{proof}[First proof of Theorem~\ref{Ber-minimal-thm}]
Theorem~4a in~\cite[p.~108]{widder} reads that a necessary and sufficient condition that the sequence $\{\mu_n\}_0^\infty$ should have the expression
\begin{equation}\label{mu=alpha-moment}
\mu_n=\int_0^1t^n\td\alpha(t)
\end{equation}
for $n\ge0$, where $\alpha(t)$ is non-decreasing and bounded for $0\le t\le1$, is that it should be completely monotonic. Theorem~14a in~\cite[p.~164]{widder} states that a completely monotonic sequence $\{\mu_n\}_{n\ge0}$ is minimal if and only if the equality~\eqref{mu=alpha-moment} is valid for $n\ge0$ and $\alpha(t)$ is a non-decreasing bounded function continuous at $t=0$.
\par
Setting in the equality~\eqref{mu=alpha-moment}
\begin{equation}
\alpha(t)=\int_0^t\frac{1}{s\{[\ln(1/s-1)]^2+\pi^2\}}\td s
\end{equation}
for $t\in[0,1]$ and $\alpha(1)=b_1=\frac12$ yields the required complete monotonicity and minimality.
\end{proof}

\begin{proof}[Second proof of Theorem~\ref{Ber-minimal-thm}]
From~\eqref{Bernoulli-2rd-int}, it follows that for $n\in\mathbb{N}$
\begin{align*}
(-1)^{n+1}b_n&=\int_1^\infty\frac{1}{\{[\ln(t-1)]^2+\pi^2\}t^{n}}\td t\\
&=\int_1^0\frac{1}{\{[\ln(1/s-1)]^2+\pi^2\}}s^n\td\biggl(\frac1s\biggr)\\
&=\int_0^1\frac{1}{\{[\ln(1/s-1)]^2+\pi^2\}}s^{n-2}\td s\\
&=\int_0^1\frac{1}{s\{[\ln(1/s-1)]^2+\pi^2\}}s^{n-1}\td s\\
&\triangleq c_{n-1}.
\end{align*}
Since $c_0=b_1=\frac12$ and the function $\frac{1}{s\{[\ln(1/s-1)]^2+\pi^2\}}$ is positive on $(0,1)$, then the complete monotonicity and minimality of the sequence $\{c_n\}_0^\infty$ is readily obtained. The proof of Theorem~\ref{Ber-minimal-thm} is complete.
\end{proof}

\begin{proof}[Proof of Theorem~\ref{b(n)-matrix-thm}]
A function $f$ is said to be completely monotonic on an interval $I$ if $f$ has derivatives of all
orders on $I$ and $(-1)^{n}f^{(n)}(x)\ge0$ for $x\in I$ and $n\ge0$. See~\cite[Chapter~XIII]{mpf-1993} and~\cite[Chapter~IV]{widder}.
\par
From the proofs of Theorem~\ref{thm-Bernoulli-Integ-Form}, we observe that
\begin{equation}\label{b(n)-limit}
b_n=(-1)^{n+1}\lim_{x\to0^+}h_n(x)
\end{equation}
and
\begin{equation}\label{h(x)-dfn-eq}
h_n(x)=\int_1^\infty\frac{1}{\{[\ln(t-1)]^2+\pi^2\}(t+x)^{n}}\td t
\end{equation}
is completely monotonic on $[0,\infty)$.
\par
In~\cite{two-place}, or see~\cite[p.~367]{mpf-1993}, it was obtained that if $f$ is a completely monotonic function on $[0,\infty)$, then
\begin{equation}\label{matrix-eq-1}
\bigl|f^{(a_i+a_j)}(x)\bigr|_m\ge0
\end{equation}
and
\begin{equation}\label{matrix-eq-2}
\bigl|(-1)^{a_i+a_j}f^{(a_i+a_j)}(x)\bigr|_m\ge0,
\end{equation}
where $|a_{ij}|_m$ denotes a determinant of order $m$ with elements $a_{ij}$ and $a_i$ for $1\le i\le m$ are nonnegative integers. Applying $f$ in~\eqref{matrix-eq-1} and~\eqref{matrix-eq-2} to the function $h_n(x)$ yields
\begin{equation}\label{matrix-eq-1h}
\bigl|h_n^{(a_i+a_j)}(x)\bigr|_m\ge0
\end{equation}
and
\begin{equation}\label{matrix-eq-2h}
\bigl|(-1)^{a_i+a_j}h_n^{(a_i+a_j)}(x)\bigr|_m\ge0,
\end{equation}
that is,
\begin{equation}\label{matrix-eq-1b}
\biggl|(-1)^{a_i+a_j}\frac{(n+a_i+a_j-1)!}{(n-1)!}h_{n+a_i+a_j}(x)\biggr|_m\ge0
\end{equation}
and
\begin{equation}\label{matrix-eq-2b}
\biggl|\frac{(n+a_i+a_j-1)!}{(n-1)!}h_{n+a_i+a_j}(x)\biggr|_m\ge0.
\end{equation}
Letting $x\to0^+$ in~\eqref{matrix-eq-1b} and~\eqref{matrix-eq-2b} and making use of~\eqref{b(n)-limit} produce
\begin{equation}\label{matrix-1b}
\biggl|(-1)^{a_i+a_j}\frac{(n+a_i+a_j-1)!}{(n-1)!}(-1)^{n+a_i+a_j+1}b_{n+a_i+a_j}\biggr|_m\ge0
\end{equation}
and
\begin{equation}\label{matrix-2b}
\biggl|\frac{(n+a_i+a_j-1)!}{(n-1)!}(-1)^{n+a_i+a_j+1}b_{n+a_i+a_j}\biggr|_m\ge0.
\end{equation}
Further simplifying~\eqref{matrix-1b} and~\eqref{matrix-2b} leads to
\begin{equation*}
\bigl|(-1)^{n+1}{(n+a_i+a_j-1)!}b_{n+a_i+a_j}\bigr|_m\ge0
\end{equation*}
and
\begin{equation*}
\bigl|(-1)^{n+a_i+a_j+1}{(n+a_i+a_j-1)!}b_{n+a_i+a_j}\bigr|_m\ge0,
\end{equation*}
which are equivalent to~\eqref{matrix-1f} and~\eqref{matrix-2f}. Theorem~\ref{b(n)-matrix-thm} is thus proved.
\end{proof}

\begin{proof}[Proof of Theorem~\ref{lambda-mu-thm}]
In~\cite[p.~106, Theorem~A]{haerc1} and~\cite[p.~367, Theorem~2]{mpf-1993}, a minor correction of~\cite[Theorem~1]{finkjmaa82}, it was obtained that if $f$ is a completely monotonic function on $(0,\infty)$ and $\lambda\preceq\mu$, then
\begin{equation}\label{finkjmaa82=ineq3.2}
\Biggl|\prod_{i=1}^nf^{(\lambda_i)}(x)\Biggr|\le \Biggl|\prod_{i=1}^nf^{(\mu_i)}(x)\Biggr|.
\end{equation}
Applying the inequality~\eqref{finkjmaa82=ineq3.2} to $h_n(x)$, defined by~\eqref{h(x)-dfn-eq}, creates
\begin{equation*}
\Biggl|\prod_{i=1}^m(-1)^{\lambda_i} \frac{(n+\lambda_i-1)!}{(n-1)!}h_{n+\lambda_i}(x)\Biggr|
\le \Biggl|\prod_{i=1}^m (-1)^{\mu_i}\frac{(n+\mu_i-1)!}{(n-1)!}h_{n+\mu_i}(x)\Biggr|
\end{equation*}
which can be simplified as
\begin{equation*}
\Biggl|\prod_{i=1}^m {(n+\lambda_i-1)!}h_{n+\lambda_i}(x)\Biggr|
\le \Biggl|\prod_{i=1}^m {(n+\mu_i-1)!}h_{n+\mu_i}(x)\Biggr|.
\end{equation*}
Further taking $x\to0^+$ and utilizing~\eqref{b(n)-limit} turn out
\begin{equation*}
\Biggl|\prod_{i=1}^m {(n+\lambda_i-1)!}(-1)^{n+\lambda_i+1}b_{n+\lambda_i}\Biggr|\le \Biggl|\prod_{i=1}^m {(n+\mu_i-1)!}(-1)^{n+\mu_i+1}b_{n+\mu_i}\Biggr|
\end{equation*}
which is equivalent to~\eqref{lambda-mu-eq}. The proof of Theorem~\ref{lambda-mu-thm} is complete.
\end{proof}

\begin{proof}[Proof of Corollary~\ref{b(n)-log-Convex-thm}]
It is clear that $(i,i+2)\succ(i+1,i+1)$ for $i\ge0$. Therefore, by virtue of~\eqref{lambda-mu-eq}, we have
\begin{equation}\label{b(n)-log-convex-eq}
(i!b_{i+1})[(i+2)!b_{i+3}]\ge [(i+1)!b_{i+2}]^2.
\end{equation}
This implies the required logarithmic convexity.
\par
This conclusion can also be deduced from Theorem~\ref{b(n)-matrix-thm}.
The proof of Theorem~\ref{b(n)-log-Convex-thm} is thus complete.
\end{proof}

\section{Proofs of Lemma~\ref{thm-zhang-li-qi}}\label{sec-proof}

Now we are in a position to prove Lemma~\ref{thm-zhang-li-qi} as follows.

\begin{proof}[First proof]
For $z= \varepsilon e^{\theta i}$ with $\theta\in\bigl[-\frac\pi2,\frac\pi2\bigr]$ and $ \varepsilon\in(0,1)$, by standard argument, we have
\begin{equation*}
|zF(z-1)|^2=\biggl|\frac{\varepsilon e^{\theta i}-1}{\ln(\varepsilon e^{\theta i})}\biggr|^2
=\frac{1-2\varepsilon \cos\theta+\varepsilon^2}{(\ln\varepsilon)^2+\theta^2}\to0
\end{equation*}
uniformly as $\varepsilon\to0^+$. Consequently,
\begin{equation}\label{zF(z)=0}
 \lim_{ \varepsilon\to0^+}[zF(z-1)]=0
\end{equation}
uniformly.
\par
For $\theta\in(-\pi,\pi)$ and $z=re^{\theta i}$, by standard argument, we have
\begin{equation}\label{F(z)-infinity}
|F(z-1)|=\biggl|\frac{re^{\theta i}-1}{re^{\theta i}\ln(re^{\theta i})}\biggr|
=\sqrt{\frac{1+2r\cos\theta+r^2}{r^2[(\ln r)^2+\theta^2]}}\,\to0
\end{equation}
uniformly as $r\to\infty$.
\par
For $t\in(0,\infty)$ and $\varepsilon\in(0,1)$, we have
\begin{align*}
F(-t-1+\varepsilon i)&=\frac{-t-1+\varepsilon i}{(-t+\varepsilon i)\ln(-t+\varepsilon i)}\\
&=\frac{-t-1+\varepsilon i}{(-t+\varepsilon i)[\ln|-t+\varepsilon i|+i\arg(-t+\varepsilon i)]}\\
&=\frac{-t-1+\varepsilon i}{(-t+\varepsilon i)\bigl[\ln|-t+\varepsilon i| +i\bigl(\pi-\arctan\frac{\varepsilon}t\bigr)\bigr]}\\
&\to\frac{t+1}{t(\ln t +\pi i)}\\
&=\frac{(t+1)(\ln t -\pi i)}{t[(\ln t)^2 +\pi^2]}
\end{align*}
as $\varepsilon\to0^+$. In other words, for $t\in(0,\infty)$,
\begin{equation}\label{im-limit}
\lim_{\varepsilon\to0^+}\Im F(-t-1+\varepsilon i)=-\frac{\pi(t+1)}{t[(\ln t)^2 +\pi^2]}.
\end{equation}
\par
Let $D$ be a bounded domain with piecewise smooth boundary. If $f(z)$ is analytic on $D$ and extendable smoothly to the boundary of $D$, then
\begin{equation}\label{cauchy-formula}
f(z)=\frac1{2\pi i}\oint_{\partial D}\frac{f(w)}{w-z}\td w,\quad z\in D.
\end{equation}
In the literature, we call~\eqref{cauchy-formula} Cauchy integral formula. See~\cite[p.~113]{Gamelin-book-2001}.
For any fixed point $z_0=x_0+iy_0\in\mathbb{C}\setminus(-\infty,0]$, choose $\varepsilon$ and $r$ such that \begin{equation*}
\begin{cases}
0<\varepsilon<|y_0|\le|z_0|<r, & y_0\ne0,\\
0<\varepsilon<x_0=|z_0|<r, & y_0=0,\\
\end{cases}
\end{equation*}
and consider the positively oriented contour $C(\varepsilon,r)$ in $\mathbb{C}\setminus(-\infty,-1]$ consisting of the half circle $z=-1+\varepsilon e^{\theta i}$ for $\theta\in\bigl[-\frac\pi2,\frac\pi2\bigr]$ and the half lines $z=-1+x\pm \varepsilon i$ for $x\le0$ until they cut the circle $|z+1|=r$, which close the contour at the points $-1-r(\varepsilon)\pm \varepsilon i$, where $0<r(\varepsilon)\to r$ as $\varepsilon\to0$.
By the formula~\eqref{cauchy-formula}, we have
\begin{multline}\label{F(z)-Cauchy-Apply}
F(z_0)=\frac1{2\pi i}\biggl[\int_{\pi/2}^{-\pi/2}\frac{i\varepsilon e^{\theta i}F\bigl(\varepsilon e^{\theta i}-1\bigr)} {\varepsilon e^{\theta i}-1-z_0}\td\theta +\int_{-r(\varepsilon)}^0 \frac{F(x-1+\varepsilon i)}{x-1+\varepsilon i-z_0}\td x \\
+\int_0^{-r(\varepsilon)}\frac{F(x-1-\varepsilon i)}{x-1-\varepsilon i-z_0}\td x +\int_{\arg[-r(\varepsilon)-\varepsilon i]}^{\arg[-r(\varepsilon)+\varepsilon i]}\frac{ir e^{\theta i}F\bigl(re^{\theta i}-1\bigr)}{re^{\theta i}-1-z_0}\td\theta\biggr].
\end{multline}
By the formula~\eqref{zF(z)=0}, it follows that
\begin{equation}\label{zf(z)=0}
\lim_{\varepsilon\to0^+}\int_{\pi/2}^{-\pi/2}\frac{i\varepsilon e^{\theta i}F\bigl(\varepsilon e^{\theta i}-1\bigr)} {\varepsilon e^{\theta i}-1-z_0}\td\theta=0.
\end{equation}
In virtue of the limit~\eqref{F(z)-infinity}, it can be derived that
\begin{equation}\label{big-circle-int=0}
\begin{split}
&\quad\lim_{\substack{\varepsilon\to0^+\\r\to\infty}} \int_{\arg[-r(\varepsilon)-\varepsilon i]}^{\arg[-r(\varepsilon)+\varepsilon i]}\frac{ir e^{\theta i}F\bigl(re^{\theta i}-1\bigr)}{re^{\theta i}-1-z_0}\td\theta\\
&=\lim_{r\to\infty}\int_{-\pi}^{\pi}\frac{ir e^{\theta i}F\bigl(re^{\theta i}-1\bigr)}{re^{\theta i}-1-z_0}\td\theta\\
&=0.
\end{split}
\end{equation}
Making use of the obvious fact that $F(\overline{z})=\overline{F(z)}$ and the limit~\eqref{im-limit} yields that
\begin{align}
&\quad\int_{-r(\varepsilon)}^0 \frac{F(x-1+\varepsilon i)}{x-1+\varepsilon i-z_0}\td x
+\int_0^{-r(\varepsilon)}\frac{F(x-1-\varepsilon i)}{x-1-\varepsilon i-z_0}\td x \notag\\
&=\int_{-r(\varepsilon)}^0 \biggl[\frac{F(x-1+\varepsilon i)}{x-1+\varepsilon i-z_0} -\frac{F(x-1-\varepsilon i)}{x-1-\varepsilon i-z_0}\biggr]\td x\notag\\
&=\int_{-r(\varepsilon)}^0\frac{(x-1-\varepsilon i-z_0)F(x-1+\varepsilon i) -(x-1+\varepsilon i-z_0)F(x-1-\varepsilon i)} {(x-1+\varepsilon i-z_0)(x-1-\varepsilon i-z_0)}\td x\notag\\
&=2i\int_{-r(\varepsilon)}^0\frac{(x-1-z_0)\Im F(x-1+\varepsilon i) -\varepsilon\Re F(x-1+\varepsilon i)} {(x-1+\varepsilon i-z_0)(x-1-\varepsilon i-z_0)}\td x\notag\\
&\to2i\int_{-r}^0\frac{\lim_{\varepsilon\to0^+}\Im F(x-1+\varepsilon i)}{x-1-z_0}\td x\notag\\
&=-2i\int^r_0\frac{\lim_{\varepsilon\to0^+}\Im F(-t-1+\varepsilon i)}{t+1+z_0}\td t\notag\\
&\to-2i\int^\infty_0\frac{\lim_{\varepsilon\to0^+}\Im F(-t-1+\varepsilon i)}{t+1+z_0}\td t\notag\\
&=2\pi i\int_0^\infty \frac{t+1}{t[(\ln t)^2 +\pi^2]}\frac{\td t}{t+1+z_0} \label{level0lines}
\end{align}
as $\varepsilon\to0^+$ and $r\to\infty$. Substituting equations~\eqref{zf(z)=0}, \eqref{big-circle-int=0}, and~\eqref{level0lines} into~\eqref{F(z)-Cauchy-Apply} and simplifying produce the integral representation~\eqref{thm-zhang-li-qi-eq}. The proof of Lemma~\ref{thm-zhang-li-qi} is complete.
\end{proof}

\begin{proof}[Second proof]
In all treatments of Pick functions, a main example is the principal logarithm $\ln$ defined in the cut plane $\mathcal{A}$ as well as
\begin{equation}
-\frac1{\ln z}=-\frac1{z-1}+\int_{-\infty}^0\frac1{(t-z)[(\ln t)^2+\pi^2]}\td t.
\end{equation}
This formula is equivalent to~\cite[(1.4)]{Berg-JIPAM-2001}. Multiplying the identity
\begin{equation}
\int_0^\infty\frac1{t[(\ln t)^2+\pi^2]}=1
\end{equation}
by $\frac1z$ and inserting it in the previous formula yield
\begin{equation}
\frac{z-1}{z\ln z}=\int_0^\infty\biggl[\frac1{tz}+\frac{z-1}{z(t+z)}\biggr]\frac{\td t}{(\ln t)^2+\pi^2}
=\int_0^\infty\frac{1+t}{(t+z)[(\ln t)^2+\pi^2]}\td t,
\end{equation}
which is the formula~\eqref{thm-zhang-li-qi-eq}. The proof of Lemma~\ref{thm-zhang-li-qi} is complete.
\end{proof}

\section{Remarks}

Finally, we would like to give some remarks on something related to the integral representations~\eqref{recip-ln(1+z)} and~\eqref{thm-zhang-li-qi-eq}.

\begin{rem}
In~\cite[p.~230, 5.1.32]{abram}, it is listed that
\begin{equation}\label{ln-frac}
\ln\frac{b}a=\int_0^\infty\frac{e^{-au}-e^{-bu}}u\td u.
\end{equation}
As a result, we have
\begin{equation}\label{lnln-int}
\ln[\ln(1+x)]=\int_0^\infty\frac{e^{-u}-e^{-u\ln(1+x)}}u\td u =\int_0^\infty\frac{e^{-u}-(1+x)^{-u}}u\td u
\end{equation}
and, by a differentiation,
\begin{equation}
\begin{aligned}\label{frac1-ln-x+1-int}
\frac1{(1+x)\ln(1+x)}&=\int_0^\infty\frac1{(1+x)^{u+1}}\td u\\
&=\int_0^\infty\biggl[\frac1{\Gamma(1+u)}\int_0^\infty t^ue^{-(1+x)t}\td t\biggr]\td u\\
&=\int_0^\infty\biggl[\int_0^\infty\frac{t^u}{\Gamma(1+u)}\td u\biggr]e^{-(1+x)t}\td t,
\end{aligned}
\end{equation}
where $\Gamma(z)$ is the classical gamma function which may be defined for $\Re(z)>0$ by Euler's integral
\begin{equation}
\Gamma(z)=\int_0^\infty t^{z-1}e^{-t}\td t.
\end{equation}
The integral representation~\eqref{frac1-ln-x+1-int} means that $\frac1{(1+x)\ln(1+x)}$ is a completely monotonic function on $(0,\infty)$. In other words, the function $\frac1{\ln(1+x)}$ is logarithmically completely monotonic on $(0,\infty)$. More strongly, it was claimed in~\cite[p.~2130, (34)]{Berg-Pedersen-ball-PAMS} and~\cite[p.~12, (33)]{Berg-Pedersen-arXiv-0912.2185} that the function $\frac1{\ln(1+x)}$ is a Stieltjes transform. For information on the notions ``logarithmically completely monotonic function'' and ``Stieltjes transform'', please refer to~\cite[Remark~8]{subadditive-qi-guo-jcam.tex}, \cite[Section~1]{SCM-2012-0142.tex}, \cite[Remark~4.7]{Open-TJM-2003-Banach.tex}, the monograph~\cite{Schilling-Song-Vondracek-2nd} and plenty of closely-related references therein.
\par
From~\eqref{frac1-ln-x+1-int} and by integration by part, it is not difficult to obtain that
\begin{equation}\label{frac1-ln-x-int}
\frac1{\ln(1+x)}=\int_0^\infty\biggl[\int_0^\infty\frac{t^{u-1}} {\Gamma(u)}\td u\biggr]e^{-(1+x)t}\td t, \quad x>0.
\end{equation}
By induction and integration by part, we can obtain
\begin{equation}
\begin{aligned}\label{frac1ln-(1+x)-int}
\frac{(1+x)^k}{\ln(1+x)}&=\int_0^\infty\biggl[\int_0^\infty\frac{t^{u-k-1}} {\Gamma(u-k)}\td u\biggr]e^{-(1+x)t}\td t\\
&=\int_0^\infty\biggl[\int_{-k}^\infty\frac{t^{u-1}} {\Gamma(u)}\td u\biggr]e^{-(1+x)t}\td t
\end{aligned}
\end{equation}
for $x>0$ and $k\in\mathbb{Z}$, where $\mathbb{Z}$ denotes the set of all integers and the classical gamma function $\Gamma(z)$ may be defined for $z\in\mathbb{C}\setminus\{0,-1,-2,\dotsc\}$ by Euler's formula
\begin{equation}\label{gamma-dfn-C}
\Gamma(z)=\lim_{n\to\infty}\frac{n!n^z}{z(z+1)\dotsm(z+n)}.
\end{equation}
\end{rem}

\begin{rem}
By the way, the term $\frac1z$ in~\eqref{recip-ln(1+z)} was lost in~\cite[p.~2130, (34)]{Berg-Pedersen-ball-PAMS} and~\cite[p.~12, (33)]{Berg-Pedersen-arXiv-0912.2185} and was corrected in~\cite{Zhang-Li-Qi-Log.tex, Zhang-Xiao-Jing-Thesis.tex}.
\end{rem}

\begin{rem}
A function $f:I\subseteq(0,\infty)\to[0,\infty)$ is called a Bernstein function on $I$ if $f(x)$ has derivatives of all orders and $f'(x)$ is completely monotonic on $I$. See the monograph~\cite{Schilling-Song-Vondracek-2nd}.
We claim that the generating function $\frac{x}{\ln(1+x)}$ of Bernoulli numbers of the second kind $b_k$ is a Bernstein function on $(0,\infty)$. This can be proved by two approaches below.
\par
The integral representation~\eqref{frac(x)(ln(1+x))-eq1} shows us that the function $\frac{x}{\ln(1+x)}$ is positive and increasing on $(0,\infty)$. The integral representation~\eqref{deriv-int-expression} reveals that the first derivative of $\frac{x}{\ln(1+x)}$ is completely monotonic on $(0,\infty)$. So the function $\frac{x}{\ln(1+x)}$ is a Bernstein function on $(0,\infty)$.
\par
It is not difficult to see that
\begin{equation}
\frac{x}{\ln(1+x)}=\int_0^1(1+x)^t\td t
\end{equation}
and the function $(1+x)^t$ for $t\in(0,1)$ is a Bernstein function.
\end{rem}

\begin{rem}
This paper is a combined and revised version of the preprints~\cite{derivative-reciprocal-log.tex, Zhang-Li-Qi-Log.tex} and Chapter~5 of the thesis~\cite{Zhang-Xiao-Jing-Thesis.tex}.
\end{rem}

\end{document}